\documentclass[12pt]{amsart}
\usepackage{amsfonts}
\usepackage{amssymb}
\usepackage{times}
\usepackage{xcolor}
\usepackage{enumitem}
\usepackage{amsfonts}

\setlist[itemize,1]{label=$\bullet$}
\setlist[itemize,2]{label=$\triangleright$}
\setlist[itemize,4]{label=$\diamond$}

\newcommand{\black}{\color{black}}

{\theoremstyle{plain}%
  \newtheorem{theorem}{Theorem}[section]
  \newtheorem{corollary}[theorem]{Corollary}
  \newtheorem{proposition}[theorem]{Proposition}
  \newtheorem{lemma}[theorem]{Lemma}

  \newtheorem{definition}[theorem]{Definition}

\newfont{\hueca}{msbm10}
\def\hu #1{\hbox{\hueca #1}}\def\hu #1{\hbox{\hueca #1}}

\begin{document}

\title{Split Malcev-Poisson-Jordan algebras}

\thanks{The first author was   supported by the Centre for Mathematics of the University of Coimbra - UIDB/00324/2020, funded by the Portuguese Government through FCT/MCTES. The second author acknowledges the support by the PCI of the UCA `Teor\'\i a de Lie y Teor\'\i a de Espacios de Banach' and by the PAI with project number FQM298.}

\author[E. Barreiro]{Elisabete~Barreiro}

\address{Elisabete~Barreiro, University of Coimbra, CMUC, Department of Mathematics, Apartado 3008 EC Santa Cruz 3001-501 Coimbra, Portugal. \hspace{0.1cm} {\em E-mail address}: {\tt mefb@mat.uc.pt}}{}

\author[Jos\'{e} M. S\'{a}nchez]{Jos\'{e} M. S\'{a}nchez}

\address{Jos\'{e} M. S\'{a}nchez, Departament of Mathematics, University of Cadiz, Puerto Real (Cadiz). Espa\~na. 
\hspace{0.1cm} {\em E-mail address}: 
{\tt txema.sanchez@uca.es}}{}

\begin{abstract}
We introduce the class of split Malcev-Poisson-Jordan algebras as the natural extension of the one of split Malcev Poisson algebras, and therefore split (non-commutative) Poisson algebras. We show that a split Malcev-Poisson-Jordan algebra $P$  can be written as a direct sum $P = \oplus_{j \in J}I_j$ with any $I_j$ a non-zero ideal of $P$ in such a way that satisfies $[I_{j_1},I_{j_2}] = I_{j_1} \circ I_{j_2} = 0$ for $j_1 \neq j_2.$ Under certain conditions, it is shown that the above decomposition of $P$ is by means of the family of its simple ideals.

\medskip



\medskip

\end{abstract}

\maketitle

%%%%%%%%%%%%%%%%%%%%%%%%%%%%%%%%%%%%%%%%%%%%%%%%%%%%%%%%%%%%%%%%
%%%%%%%%%%%%%%%%%%%%%%%%%%%%%%%%%%%%%%%%%%%%%%%%%%%%%%%%%%%%%%%%
\section{Introduction and previous definitions}
%%%%%%%%%%%%%%%%%%%%%%%%%%%%%%%%%%%%%%%%%%%%%%%%%%%%%%%%%%%%%%%%
%%%%%%%%%%%%%%%%%%%%%%%%%%%%%%%%%%%%%%%%%%%%%%%%%%%%%%%%%%%%%%%%
Recently Malcev-Poisson-Jordan algebras were presented in \cite{MPJ} as a natural reformulation of Poisson algebras endowed with a Malcev structure and Jordan conditions (instead of the Lie structure and associativity, respectively) related by a Leibniz identity. Furthermore, in the referred  paper the authors presented Malcev-Poisson-Jordan structures on some classes of Malcev algebras, also study the concept of pseudo-Euclidean Malcev-Poisson-Jordan algebras and nilpotent pseudo-Euclidean Malcev-Poisson-Jordan algebras are defined.

In the present paper we introduce the class of split Malcev-Poisson-Jordan algebras as the natural extension of the one of split (non-commutative) Poisson algebras and study its structure. We define connections on the root system of a split Malcev-Poisson-Jordan algebra $P$ that becomes an equivalence relation. We prove that  $P$ can be expressed as $P = \oplus_{j \in J}I_j$ with any $I_j$ a non-zero ideal of $P$ in such a way that satisfies $[I_{j_1},I_{j_2}] = I_{j_1} \circ I_{j_2} = 0$ for $j_1 \neq j_2,$ and under additional  hypothesis these components of $P$ are simple.

The paper is structured as follows. In this section we present some basic tools needed in the sequel. Section 2 introduces the connection techniques in the framework of Malcev-Poisson-Jordan algebras and it applies these techniques to study the structure of split Malcev-Poisson-Jordan algebras. Under certain conditions, in  Section 3 it is shown that the previous result can be improved with a direct sum of simple components.

A \emph{Malcev algebra}  is an algebra $(M,[\cdot,\cdot])$ over a field $\mathbb{K}$ whose multiplication  satisfies:
\begin{itemize}
\item $[x,x]=0,$
\item $[J(x,y,z),x]=J(x,y,[x,z]),$
\end{itemize}
for all $x, y, z \in M$, where as usual $J(x,y,z):=[[x,y],z]-[[x,z],y]-[x,[y,z]]$ is the {\it Jacobian of} $x,y,z$. %The following useful identity can be easily deduced: \begin{equation*} [[x,y],[x,z]]=[[[x,y],z],x]+[[[y,z],x],x]+[[[z,x],x],y].\end{equation*}

A \emph{Jordan algebra} is an algebra $P$ over a field $\mathbb{K}$ whose multiplication, defined by juxtaposition, satisfies the following axioms:
\begin{itemize}
\item $xy = yx,$
\item $(x^2 y)x = x^2(yx),$
\end{itemize}
 for all $x, y \in P$.
%If the characteristic is not equal to 2, the following useful identity can be easily deduced (see \cite{Jacobson}): \begin{equation*} ((x y b)z  + ((y z)b)x  +((z x)b)y  = (x y)(b z) + (y z)(b x) + (z x)(b y), \end{equation*} whenever $x, y, z, b \in P$.

\begin{definition}\rm
A {\it Malcev-Poisson-Jordan algebra} $P$ (also called {\it MPJ-algebra} for short) is a $\mathbb{K}$-vector space equipped with two bilinear multiplications such that  $(P,[\cdot,\cdot])$ is a Malcev algebra, $P$ with multiplication defined by juxtaposition is a Jordan algebra and these two operations are related by the {\it Leibniz condition}:
\begin{equation}\label{Leibniz_identity}
[x,y z] = [x,y]z + y[x,z],
\end{equation}
for any $x,y,z \in P$.
\end{definition}

\noindent Poisson algebras as well as Malcev algebras  are MPJ-algebras, considering always the other product zero. Also, a  MPJ-algebra  is a Malcev Poisson algebra if the multiplication defined by juxtaposition is associative. So the present paper extends the works presented in \cite{YoPoisson, YoMalcev,  Malcev_Poisson}.

A linear subspace $S$ of a MPJ-algebra $P$ is called  a {\it subalgebra}    of $P$ if $[S,S] \subset S$ and $S  S \subset S$. A subalgebra $S$ of a MPJ-algebra $P$ is {\it abelian} if $S$ satisfies $[S,S] = SS = 0.$ An {\it ideal } of a MPJ-algebra $P$ is a subalgebra $I$ of $P$ satisfying $[I,P] \subset I$ and $IP \subset I$.

In order to study the structure of MPJ-algebras of arbitrary dimension and over an arbitrary base field, we introduce the concept of split MPJ-algebra as a MPJ-algebra in  which the underlying Malcev algebra structure is split. So, let us recall the concept of a split Malcev algebra (see \cite{YoMalcev}).
Let  $H$ be  a maximal abelian subalgebra (MASA) of a Malcev
algebra $(M,[\cdot,\cdot])$. As usual $H^*$ denotes the dual space of $H$. For a linear functional $\alpha : H \longrightarrow
\hu{K},$ we define the {\rm root space} of $M$ (with respect to $H$)
associated to $\alpha$ as the subspace $$M_{\alpha}=\Bigl\{v_{\alpha}\in M:[h,v_{\alpha}]=\alpha(h)v_{\alpha}\hspace{0.2cm} {\it for}
\hspace{0.2cm} {\it any} \hspace{0.2cm} h \in H\Bigr\}.$$ 
 The elements
$\alpha \in H^*$ satisfying $M_{\alpha} \neq 0$ are called {\rm
roots} of $M$ (with respect to $H$). We denote $\Lambda:=\{\alpha \in
H^*\setminus \{0\}:  M_{\alpha}\neq 0\}$ and say that $\Lambda$ is the {\it root system} of $M$. We say that  $M$ is a
  \emph{split Malcev algebra} (with respect to $H$) if $$M = H \oplus (\bigoplus\limits_{\alpha \in \Lambda}M_{\alpha}).$$

\begin{definition}\rm
A {\it split Malcev-Poisson-Jordan algebra} is a Malcev-Poisson-Jordan algebra in which the Malcev algebra $(P, [\cdot, \cdot])$ is split with respect to a MASA $H$ of $(P, [\cdot, \cdot])$, and  $H$ is also an abelian subalgebra of the MPJ-algebra $P$.
\end{definition}

\noindent This means that we can decompose $P$ as the direct sum $$P = H \oplus (\bigoplus\limits_{\alpha \in \Lambda} P_{\alpha})$$ where 
$$ P_{\alpha } := \Bigl\{v_{\alpha} \in P : [h,v_{\alpha}] = \alpha(h)v_{\alpha}  \mbox{ for any } h \in H \Bigr\},$$ for a linear functional $\alpha \in H^*$ and $\Lambda := \{\alpha \in H^* \setminus \{0\} : P_{\alpha} \neq 0\}$ is the corresponding {\it root system}.
The elements $\alpha \in \Lambda \cup \{0\}$ are called {\it roots} of $P$ with respect to $H$, the subspaces $P_{\alpha}$ are called {\it root spaces} of $P$ (with respect to $H$).

It is clear that the root space $P_0$ of the zero root  contains  $H$. Conversely, given any $v_0 \in {\mathcal P}_0$ we can write $v_0 = h + \sum_{i=1}^n v_{\alpha_i}$ with $h \in H$ and $v_{\alpha_i} \in P_{\alpha_i}$ for $i=1,\ldots ,n$, being $\alpha_i \in \Lambda$ with $\alpha_i \neq \alpha_j$ if $i \neq j$. Hence
$0 = [h',h+ \sum_{i=1}^n v_{\alpha_i}] = \sum_{i=1}^n \alpha_i(h')v_{\alpha_i}$ for any $h' \in H$. So, taking into account the direct character of the sum and that $\alpha_i \neq 0$, we have that any $v_{\alpha_i}=0$ and then $v_0 \in H$. Consequently
\begin{equation}\label{H}
H = P_0.
\end{equation}

It is worth to mention that throughout this paper Malcev-Poisson-Jordan algebras $P$ are considered of arbitrary dimension and over an arbitrary base field ${\hu K}$ of characteristic different from 2.  We also note that, unless otherwise stated, there is no restrictions on $\dim P_{\alpha}$, the products $[P_{\alpha},P_{-\alpha}]$, $P_{\alpha} P_{-\alpha}$ and $\{k \in \hu{K}: k\alpha \in \Lambda\}$.

\begin{lemma}\label{elprimero}
Let $P $ be a split MPJ-algebra and $ \alpha, \beta \in \Lambda \cup \{0\}$. Then the following assertions hold.
\begin{enumerate}
\item[{\rm i )}] If $[P_{\alpha},P_{\beta}] \neq 0$ with $\beta \neq \alpha$ then  $ \alpha +\beta \in \Lambda \cup \{0\}$ and  $[P_{\alpha},P_{\beta}] \subset P_{\alpha+\beta}$.
\item[{\rm ii)}] If $[P_{\alpha},P_{\alpha}]  \neq 0 $ then $2\alpha \in \Lambda \cup \{0\}$ or $ -\alpha  \in \Lambda \cup \{0\}$ and  $[P_{\alpha},P_{\alpha}] \subset P_{2\alpha} + P_{-\alpha},$ where $  P_{2\alpha} + P_{-\alpha} $  means the semidirect sum of  $ P_{2\alpha} $ and $ P_{-\alpha}.$
\item[{\rm iii)}] If $ P_{\alpha}  P_{\beta} \neq 0$ then  $ \alpha +\beta \in \Lambda \cup \{0\}$ and $P_{\alpha}  P_{\beta} \subset P_{\alpha + \beta}$.
%, by Lemma 1.1 in \cite{YoPoisson}.
\end{enumerate}
\end{lemma}

\begin{proof} First two items can be proved as in the finite dimensional case (see \cite{YoMalcev} and refe\-rences therein). The third  can be verified by taking into account \eqref{Leibniz_identity}. In fact,
for any $h \in H$, $v_{\alpha} \in P_{\alpha}, $ and $  v_{\beta} \in P_{\beta}$  we have,
\begin{eqnarray*}
\begin{split}
[h, v_{\alpha} v_{\beta}] &=
[h, v_{\alpha} ]v_{\beta}+v_{\alpha}[h,  v_{\beta}]  =  \alpha(h)v_{\alpha} v_{\beta}+\beta(h)v_{\alpha} v_{\beta} = (\alpha  +\beta)(h)v_{\alpha} v_{\beta},
\end{split}
\end{eqnarray*}
therefore $v_{\alpha} v_{\beta} \in P_{\alpha + \beta}$.
\end{proof}
\noindent Observe that Equation \eqref{H} and Lemma \ref{elprimero}-iii) imply $$\hbox{$H  P_{\alpha} + P_{\alpha}  H \subset P_{\alpha }$, for any $\alpha \in
\Lambda$.}$$

%%%%%%%%%%%%%%%%%%%%%%%%%%%%%%%%%%%%%%%%%%%%%%%%%%%%%%%%%%%%%%%%
%%%%%%%%%%%%%%%%%%%%%%%%%%%%%%%%%%%%%%%%%%%%%%%%%%%%%%%%%%%%%%%%
\section{Connections of Roots. Decompositions}
%%%%%%%%%%%%%%%%%%%%%%%%%%%%%%%%%%%%%%%%%%%%%%%%%%%%%%%%%%%%%%%%
%%%%%%%%%%%%%%%%%%%%%%%%%%%%%%%%%%%%%%%%%%%%%%%%%%%%%%%%%%%%%%%%

In the following $P$ denotes a split Malcev-Poisson-Jordan algebra with root system $\Lambda$, and $P = H \oplus (\bigoplus_{\alpha \in \Lambda} P_{\alpha})$ its corresponding root decomposition.  Given a linear functional $\alpha : H \to \mathbb{K},$ we denote by $-\alpha : H \to \mathbb{K}$ the form in $H^*$ defined by $(-\alpha)(h) := -\alpha(h)$, for any $h \in H.$  We also represent $-\Lambda := \{-\alpha : \alpha \in \Lambda\}$ and $\pm \Lambda := \Lambda \cup -\Lambda.$
%We begin by developing connections of roots techniques in this framework, but previously we define the sets $-\Lambda := \{-\alpha : \alpha \in \Lambda\},$ where $-\alpha$ means the opposite mapping to $\alpha : H \to \mathbb{K},$ and $\pm \Lambda := \Lambda \dot{\cup} -\Lambda.$
Let us denote by $$\Omega := \Bigl\{\alpha \in \Lambda : [P_{\alpha},P_{-\alpha}] \neq 0 \Bigr\} \cup \Bigl\{\alpha \in \Lambda : [[P_{\beta}, P_{-\beta}],P_{\alpha}] \neq 0 \hspace{0.1cm} {\rm for} \hspace{0.1cm} {\rm some} \hspace{0.1cm} \beta \in \Lambda \Bigr\}$$ $$\cup \Bigl\{\alpha \in \Lambda : P_{\alpha}P_{-\alpha} \neq 0 \Bigr\}  $$  

\noindent Since $[P_{\beta}, P_{-\beta}] \subset H$, the skew-symmetry of the Malcev product joint with  $\alpha([P_{\beta}, P_{-\beta}]) \neq 0$ implies $-\alpha([P_{\beta}, P_{-\beta}]) \neq 0$, we obtain that
\begin{equation}\label{equmenos1}
\hbox{if $\alpha \in \Omega$ and $-\alpha \in \Lambda$ then $-\alpha \in \Omega.$}
\end{equation}
We associate to any $\alpha \in \Omega$ the symbol $\theta_{\alpha}$ and denote $\Theta_{\Omega} := \{\theta_{\alpha} : \alpha \in \Omega\}.$
Let us define the mapping $$\star : (\pm \Lambda \cup \Theta_{\Omega}) \times \pm \Lambda \to H^* \cup \Theta_{\Omega},$$ as follows:
\begin{itemize}
\item For $\alpha \in \pm \Lambda, \alpha \star (-\alpha) := \left\{%
\begin{array}{ll}
\theta_{\alpha}, & \hbox{if $\alpha \in \Omega$} \\
  0, & \hbox{if $\alpha \notin \Omega$} \\
\end{array}%
\right. $

\item For $\alpha, \beta \in \pm \Lambda$ with $\beta \neq -\alpha$, we define $\alpha \star \beta \in H^*$ as the usual sum of linear functionals, that is $(\alpha \star \beta)(h) := (\alpha + \beta)(h)=\alpha(h)+ \beta(h)$ for any $h \in H.$

\item For $ \alpha \in  \Omega$ and $\beta \in \pm \Lambda$, $$\theta_{\alpha}\star\beta := \left\{%
\begin{array}{ll}
 \beta, & \mbox{if either} [[P_{\alpha}, P_{-\alpha}], P_{\beta}] \neq 0 \mbox{ or }  (P_{\alpha}, P_{-\alpha}) P_{\beta} \neq 0 \\
 & \mbox{or }  \bigl[P_{\alpha}P_{-\alpha},P_{\beta}\bigr]\neq 0 \mbox{ or } \bigl[P_{\alpha},P_{-\alpha}\bigr]P_{\beta} \neq 0\\
 0, & \mbox{otherwise}  
 \end{array}%
\right.  $$
\end{itemize}
where $0 : H \to {\hu K}$ denotes the zero root.

\noindent Note that Equation \eqref{equmenos1} implies that
\begin{equation*}
{\hbox{if $\alpha\star (-\alpha ) = \theta_{\alpha}$ then
$(-\alpha)\star \alpha = \theta_{-\alpha}$.}}
\end{equation*}

\begin{definition}\label{defco}\rm
Let $\alpha$, $\beta$ be two non-zero roots. We say that
$\alpha$ is {\it connected} to  $\beta$ if there exist $\alpha
_1,\dots,\alpha_n \in \pm \Lambda$ such that
\begin{itemize}
\item[{\rm i)}] $\alpha_1 = \alpha$

\item[{\rm ii)}] $\alpha_1 \star \alpha_2 \in \pm \Lambda \cup \Theta_{\Omega},$

\hspace{-0.5cm} $ (\alpha_1 \star \alpha_2) \star\alpha_3  \in \pm \Lambda \cup \Theta_{\Omega},$

\vspace{0.05cm} $\vdots$

\hspace{-0.5cm} $(\cdots((\alpha_1 \star \alpha_2) \star\alpha_3) \star \cdots) \star \alpha_{n-1} \in \pm \Lambda \cup \Theta_{\Omega}.$

\item [{\rm iii)}]
$((\cdots((\alpha_1 \star \alpha_2) \star \alpha_3) \star \cdots) \star \alpha_{n-1}) \star \alpha_n \in \{\beta,-\beta\}$.
\end{itemize}
We also say that $\{\alpha_1,\dots,\alpha_n\}$ is a {\it connection} from $\alpha$ to $\beta$.
\end{definition}

\noindent Observe that $\{\alpha\}$ is a connection from $\alpha$ to itself and to $-\alpha.$
 The proof of the next result is analogous to the proof  of \cite[Proposition 3.1]{YoMalcev} taking into account $ \Lambda$ is symmetric in the sense defined in that work.

\begin{proposition}\label{pro1}
The relation $\sim$ in $\Lambda$, defined by $\alpha \sim \beta$
if and only if $\alpha$ is connected to $\beta,$ is of
equivalence.
\end{proposition}

\noindent Applying Proposition \ref{pro1}, we may consider the quotient set $$\Lambda / \sim := \{[\alpha] : \alpha \in \Lambda\},$$ where $[\alpha] := \{\beta \in \Lambda : \beta \sim \alpha\}.$
Our next goal is to associate an (adequate) ideal $I_{[\alpha]}$ of the split Malcev-Poisson-Jordan algebra $P$ to any $[\alpha]$ of $\Lambda/\sim.$ For each $[\alpha]$, with $\alpha \in \Lambda$, we define $$I_{H,[\alpha]} := span_{\mathbb{K}} \Bigl\{[P_{\beta},P_{-{\beta}}] + P_{\beta}  P_{-{\beta}} : \beta \in [\alpha] \Bigr\} \subset H$$ and $$V_{[\alpha]} := \bigoplus\limits_{\beta \in [\alpha]} P_{\beta}.$$ Finally, we denote by $I_{[\alpha]}$ the following linear subspace of $P$, $$I_{[\alpha]} := I_{H,[\alpha]} \oplus V_{[\alpha]}.$$

\noindent The first item of next proposition says that for any $[\alpha] \in \Lambda/\sim$, the linear subspace $I_{[\alpha]}$ is a subalgebra of $P.$ 

\begin{proposition}\label{pro2}
Let $\alpha , \gamma \in \Lambda$. Then the following assertions hold.

\begin{enumerate}
\item[{\rm i)}] $[I_{[\alpha]},I_{[\alpha]}] \subset I_{[\alpha]}$ and $I_{[\alpha]}  I_{[\alpha]} \subset I_{[\alpha]}.$

\item[{\rm ii)}] If $[\alpha] \neq [\gamma]$ then $[I_{[\alpha]}, I_{[\gamma]}] = I_{[\alpha]} I_{[\gamma]}=0$.
\end{enumerate}
\end{proposition}

\begin{proof}
i) We have
\begin{equation}\label{cero}
\Bigl[I_{H,[\alpha]} \oplus V_{[\alpha]},I_{H,[\alpha]} \oplus
V_{[\alpha]}\Bigr] \subset
\end{equation}
$$\bigl[I_{H,[\alpha]}, I_{H,[\alpha]}\bigr] + \bigl[I_{H,[\alpha]}, V_{[\alpha]}\bigr] + \bigl[V_{[\alpha]},I_{H,[\alpha]}\bigr] + \bigl[V_{[\alpha]},V_{[\alpha]}\bigr].$$
and
\begin{eqnarray}
&& \hspace{1.3cm} \Bigl(I_{H,[\alpha]} \oplus V_{[\alpha]}\Bigr)  \Bigl(I_{H,[\alpha]} \oplus V_{[\alpha]} \Bigr) \subset \nonumber \\
&& I_{H,[\alpha]} I_{H,[\alpha]} + I_{H,[\alpha]} V_{[\alpha]} + V_{[\alpha]}  I_{H,[\alpha]} + V_{[\alpha]} V_{[\alpha]}. \label{cerooo}
\end{eqnarray}
Since $I_{H,[\alpha]} \subset H$ we get
\begin{equation}\label{p1}
\hbox{$\bigl[I_{H,[\alpha]},I_{H,[\alpha]}\bigr] = 0$ and $\bigl[I_{H,[\alpha]},V_{[\alpha]}\bigr] + \bigl[V_{[\alpha]},
I_{H,[\alpha]}\bigr] \subset V_{[\alpha]}$,}
\end{equation}
so we only have to consider the summand $\bigl[V_{[\alpha]}, V_{[\alpha]}\bigr]$ in Equation \eqref{cero}. Given $\beta, \delta \in [\alpha]$ such that $[P_{\beta},P_{\delta}] \neq 0$. Assume that $\beta \neq \delta.$ If $\delta = -\beta$, then clearly $[P_{\beta},P_{\delta}] = [P_{\beta}, P_{-\beta}] \subset I_{H,[\alpha]}.$ Suppose $\delta \neq -\beta$. Taking into account that $[P_{\beta},P_{\delta}] \neq 0$ joint with Lemma \ref{elprimero}-i) ensures $\beta + \delta \in \Lambda$, we have that $\{\beta,\delta\}$ is a connection from $\beta$ to $\beta +\delta$. The transitivity of $\sim$ gives now that $\beta + \delta \in [\alpha]$ and so $[P_{\beta},P_{\delta}] \subset P_{\beta + \delta} \subset V_{[\alpha]}$. Finally, in case $\delta = \beta$ we proceed as above with Lemma \ref{elprimero}-ii) to get either $2\beta \in \Lambda$ or $-\beta \in \Lambda,$ and considering the connections $\{\beta,\beta\}$ or $\{\beta\},$ respectively, we have by transitivity that $2\beta \sim \alpha$ or $-\beta \sim \alpha,$ respectively, and again $[P_{\beta},P_{\beta}] \subset P_{2\beta} + P_{-\beta} \subset V_{[\alpha]}$. Consequently
\begin{equation}\label{eq0.5}
\bigl[V_{[\alpha]},V_{[\alpha]}\bigr] \subset I_{H,[\alpha]} \oplus V_{[\alpha]}.
\end{equation}
From Equations \eqref{cero}, \eqref{p1} and \eqref{eq0.5} we get
$$\bigl[I_{[\alpha]},I_{[\alpha]}\bigr] = \bigl[I_{H,[\alpha]} \oplus V_{[\alpha]}, I_{H,[\alpha]} \oplus V_{[\alpha]}\bigr]\subset I_{[\alpha]}.$$

Let us consider Equation \eqref{cerooo}. 
Since $I_{H,[\alpha]} \subset H$ we get $ I_{H,[\alpha]} I_{H,[\alpha]}  = 0$ (because $H$ is an abelian subalgebra of the MPJ-algebra $P$) and  taking into account Lemma \ref{elprimero}-iii),
\begin{equation}\label{p1a}
\hbox{$   I_{H,[\alpha]} V_{[\alpha]} + V_{[\alpha]}
I_{H,[\alpha]}  \subset V_{[\alpha]}$,}
\end{equation}
so we only have to consider the product $ V_{[\alpha]}  V_{[\alpha]} $ in Equation \eqref{cerooo}.
 Given $\beta, \delta \in [\alpha]$ such that $ P_{\beta} P_{\delta}  \neq 0$. If $\delta = -\beta$, then clearly $ P_{\beta} P_{\delta}  = P_{\beta}  P_{-\beta}  \subset I_{H,[\alpha]}.$ Suppose $\delta \neq -\beta$. Taking into account that $ P_{\beta} P_{\delta}  \neq 0$ joint with Lemma \ref{elprimero}-iii) ensures $\beta + \delta \in \Lambda$, we have that $\{\beta,\delta\}$ is a connection from $\beta$ to $\beta +\delta$. The transitivity of $\sim$ gives now that $\beta + \delta \in [\alpha]$ and so $ P_{\beta} P_{\delta}  \subset P_{\beta + \delta} \subset V_{[\alpha]}$. 
 Consequently
\begin{equation}\label{eq0.5a}
 V_{[\alpha]} V_{[\alpha]}  \subset I_{H,[\alpha]} \oplus V_{[\alpha]}.
\end{equation}
%%%%%%%%%%%%%%%%%%%%%%%%%%%
%%%%%%%%%%%%%%%%%%%%%%%%%
From here, Equations \eqref{cerooo}, \eqref{p1a}  and \eqref{eq0.5a} give us $$I_{[\alpha]}  I_{[\alpha]} = \bigl(I_{H,[\alpha]} \oplus V_{[\alpha]}\bigr)  \bigl(I_{H,[\alpha]} \oplus V_{[\alpha]}\bigr) \subset I_{[\alpha]}.$$

ii) Let $[\alpha] \neq [\gamma].$ Since $\bigl[I_{H,[\alpha]}, I_{H,[\gamma]}\bigr] \subset [H,H] = 0$ and $ I_{H,[\alpha]}  I_{H,[\gamma]}  \subset  H H  = 0$, we have
\begin{equation}\label{cuatro}
\Bigl[I_{H,[\alpha]} \oplus V_{[\alpha]}, I_{H,[\gamma]} \oplus V_{[\gamma]}\Bigr] \subset \bigl[I_{H,[\alpha]},V_{[\gamma]}\bigr]+ \bigl[V_{[\alpha]}, I_{H,[\gamma]} \bigr] + \bigl[V_{[\alpha]}, V_{[\gamma]}\bigr].
\end{equation}
as well
\begin{equation}\label{cuatrooo}
\Bigl(I_{H,[\alpha]} \oplus V_{[\alpha]}\Bigr) \Bigl(I_{H,[\gamma]} \oplus V_{[\gamma]}\Bigr) \subset   I_{H,[\alpha]}  V_{[\gamma]} + V_{[\alpha]}  I_{H,[\gamma]} + V_{[\alpha]}  V_{[\gamma]}.
\end{equation}

Consider the third summand $\bigl[V_{[\alpha]}, V_{[\gamma]}\bigr]$ in Equation \eqref{cuatro}, and suppose there exist $\beta \in [\alpha]$ and $\eta \in [\gamma]$ such that $[P_{\beta}, P_{\eta}] \neq 0$. As necessarily $\beta \neq -\eta$, by Lemma \ref{elprimero}-i) then $\beta+\eta \in \Lambda$. So $\{\beta,\eta,-\beta\}$ is a connection from $\beta$ to $\eta$. By transitivity of the connection relation $\alpha \in [\gamma]$, a contradiction. Hence $[P_{\beta},P_{\eta}] = 0$ and so
\begin{equation}\label{nueve}
\bigl[V_{[\alpha]},V_{[\gamma]}\bigr]=0.
\end{equation}
A similar argument, using now Lemma \ref{elprimero}-iii), can be applied to the third summand $V_{[\alpha]}  V_{[\gamma]}$ in Equation \eqref{cuatrooo} and we also conclude
\begin{equation}\label{nueveee}
V_{[\alpha]}  V_{[\gamma]}=0.
\end{equation}

Consider now the first summand $\bigl[I_{H,[\alpha]}, V_{[\gamma]}\bigr]$ in Equation \eqref{cuatro} and the first one $I_{H,[\alpha]} V_{[\gamma]}$ in Equation \eqref{cuatrooo}, and suppose there exist $\beta \in [\alpha]$ and $\eta \in [\gamma]$ such that
$$\bigl[\bigl[P_{\beta},P_{-\beta}\bigr], P_{\eta}\bigr] + \bigl[P_{\beta}P_{-\beta},P_{\eta}\bigr] + \bigl[P_{\beta},P_{-\beta}\bigr]P_{\eta}+\bigl(P_{\beta}P_{-\beta}\bigr)P_{\eta} \neq 0.$$
Suppose the first summand $\bigl[\bigl[P_{\beta},P_{-\beta}\bigr], P_{\eta}\bigr]$ is non-zero, so $\beta \in \Omega$ and $\theta_{ \beta}\star \eta = \eta$. Then $\{\beta, -\beta, \eta\}$ is a connection from $\beta$ to $\eta$ and $\eta \in [\alpha]$, a contradiction. Hence, $\bigl[\bigl[P_{\beta},P_{-\beta}\bigr], P_{\eta}\bigr] = 0.$  Similarly to the  summands $\bigl[P_{\beta}P_{-\beta},P_{\eta}\bigr] $, $ \bigl[P_{\beta},P_{-\beta}\bigr]P_{\eta} $ and  $\bigl(P_{\beta}P_{-\beta}\bigr)P_{\eta}.$
 From here,
\begin{equation}\label{trece}
\bigl[I_{H,[\alpha]},V_{[\gamma]}\bigr] = I_{H,[\alpha]}  V_{[\gamma]} = 0.
\end{equation}

%In case that $\bigl\{\bigl\{P_{\beta},P_{-\beta}\bigr\}, P_{\eta}\bigr\} \neq 0$ or $\bigl(P_{\beta}P_{-\beta}\bigr)P_{\eta} \neq 0$ \blue we have that $\eta-\beta \in \Lambda$ and we can proceed as above considering the connection $\{\beta,\eta-\beta\}$ \black to obtain $\eta \in [\alpha],$ a contradiction again. So $\bigl\{\bigl\{P_{\beta},P_{-\beta}\bigr\}, P_{\eta}\bigr\} = \bigl(P_{\beta}P_{-\beta}\bigr)P_{\eta} = 0.$ In case $\bigl\{P_{\beta}P_{-\beta},P_{\eta}\bigr\} \neq 0$ or $\bigl\{P_{\beta},P_{-\beta}\bigr\}P_{\eta} \neq 0,$ by Leibniz identity we have $P_{\beta}\bigl\{P_{-\beta},P_{\eta}\bigr\} + \bigl\{P_{\beta},P_{\eta}\bigr\}P_{-\beta} \neq 0$ or $\bigl\{P_{\beta}P_{\eta},P_{-\beta}\bigr\} + P_{\beta}\bigl\{P_{\eta},P_{-\beta}\bigr\} \neq 0,$ respectively, and arguing as above we obtain similar contradictions with either Equation \eqref{nueve} or Equation \eqref{nueveee}. From here,
%\begin{equation}\label{trece}
%\bigl\{I_{H,[\alpha]},V_{[\gamma]}\bigr\} + I_{H,[\alpha]}  V_{[\gamma]} = 0.
%\end{equation}
\noindent In a similar way we get
\begin{equation}\label{treceee}
\bigl[V_{[\alpha]},I_{H,[\gamma]}\bigr] = V_{[\alpha]} I_{H,[\gamma]}=0.
\end{equation}
 Equations (\ref{cuatro})-(\ref{treceee}) show $\bigl[I_{[\alpha]}, I_{[\gamma]}\bigr] = I_{[\alpha]} I_{[\gamma]} = 0$.
\end{proof}

\noindent Let us show that any $I_{[\alpha]}$ is actually an ideal of the MPJ-algebra $P$. Also, we recall that a MPJ-algebra $P$ is {\it simple} if $[P,P] \neq 0, P  P \neq 0$ and its only ideals are $\{0\},P$.

\begin{theorem}\label{teo1}
The following assertions hold.
\begin{enumerate}
\item[{\rm i)}] For any $[\alpha] \in \Lambda/\sim$, the linear
subspace $$I_{[\alpha]} = I_{H,[\alpha]} \oplus V_{[\alpha]}$$ of $P$ associated to $[\alpha]$ is an ideal of $P$.

\item[{\rm ii)}] If $P$ is simple, then there exists a connection from $\alpha$ to $\beta$ for any $\alpha, \beta \in \Lambda$ and $H = \sum\limits_{\alpha \in \Lambda}\bigl([P_{\alpha},P_{-\alpha}] + P_{\alpha}  P_{-\alpha}\bigr)$.
\end{enumerate}
\end{theorem}

\begin{proof}
i) Since $\bigl[I_{H,[\alpha]}, H\bigr] \subset [H,H] = 0$ and
$\bigl[V_{[\alpha]}, H\bigr] \subset V_{[\alpha]}$ we get
$\bigl[I_{[\alpha]}, H\bigr] \subset V_{[\alpha]}$. Using Proposition \ref{pro2} we obtain $$\bigl[I_{[\alpha]}, P\bigr] = \bigl[I_{[\alpha]},H \oplus \bigl(\bigoplus\limits_{\beta \in [\alpha]} P_{\beta} \bigr) \oplus \bigl(\bigoplus\limits_{\delta \notin [\alpha]} P_{\delta}\bigr)\bigr] \subset I_{[\alpha]}.$$
Now we need to prove that $I_{[\alpha]} P  \subset I_{[\alpha]}.$
First, we observe that for any $\beta \in [\alpha],$ using Lemma \ref{elprimero}-iii) and since $H$ is abelian, we have
\begin{equation}\label{p10}
(P_{\beta} P_{-\beta}) H  \subset P_0H   = 0.
\end{equation}
By Equation \eqref{Leibniz_identity} we also have
\begin{equation}
\begin{split}
[P_{\beta}, P_{-\beta}]  H &\subset [P_{\beta}, P_{-\beta}  H] + P_{-\beta}  [P_{\beta},H] \subset [P_{\beta}, P_{-\beta}] + P_{ \beta}  P_{-\beta} .\label{p11}
\end{split}
\end{equation}
From Equations \eqref{p10} and \eqref{p11} we get $I_{H,[\alpha]}  H   \subset I_{H,[\alpha]}$ and, taking into account $V_{[\alpha]}  H   \subset V_{[\alpha]}$, we conclude $$I_{[\alpha]} H  \subset I_{[\alpha]}.$$
%\blue By the above observation and Proposition \ref{pro2}-ii, we can assert $$I_{[\alpha]} P  \subset I_{[\alpha]}.$$ \black

Similarly, we get $I_{[\alpha]}H \subset V_{[\alpha]}$ since $I_{H,[\alpha]}H \subset HH = 0$ and $V_{[\alpha]}H  \subset V_{[\alpha]}$. Using again Proposition \ref{pro2} we obtain
$$I_{[\alpha]} P = I_{[\alpha]} \Bigl(H \oplus (\bigoplus\limits_{\beta \in [\alpha]}P_{\beta}) \oplus (\bigoplus\limits_{\gamma \notin [\alpha]}P_{\gamma})\Bigr) \subset I_{[\alpha]}.$$
Therefore $I_{[\alpha]}$ is an ideal of $P$.
\medskip

ii) The simplicity of $P $ implies $I_{[\alpha]} = P$. From
here, it is clear that $[\alpha]=\Lambda$ and $H=\sum\limits_{\alpha \in \Lambda} \bigl([P_{\alpha}, P_{-\alpha}] +P_{\alpha}  P_{-\alpha}\bigr)$.
\end{proof}

\begin{theorem}
For a linear complement $U$ of $span_{\mathbb K}\Bigl\{[P_{\alpha}, P_{-\alpha}] + P_{\alpha}P_{-\alpha} : \alpha \in \Lambda\Bigr\}$ in $H$, we have $$P = U + \sum\limits_{[\alpha] \in \Lambda/\sim} I_{[\alpha]},$$ where any $I_{[\alpha]}$ is one of the ideal of $P$ described in Theorem \ref{teo1}-i) satisfying $[I_{[\alpha]},I_{[\gamma]}] = I_{[\alpha]} I_{[\gamma]} = 0$ if $[\alpha] \neq [\gamma]$.
\end{theorem}

\begin{proof}
Clearly $I_{[\alpha]}$ is well-defined and, by Theorem \ref{teo1}-i), an ideal of $P$. Now, by consi\-dering a linear complement $U$ of $span_{\mathbb K}\Bigl\{[P_{\alpha}, P_{-\alpha}] + P_{\alpha}P_{-\alpha} : \alpha \in \Lambda\Bigr\}$ in $H$, we have $$P = H \oplus (\bigoplus\limits_{\alpha \in \Lambda}P_{\alpha}) = U + \sum\limits_{[\alpha ] \in \Lambda/\sim}I_{[\alpha]}.$$
Finally from Proposition \ref{pro2}-ii), it follows that $[I_{[\alpha]}, I_{[\gamma]}] = I_{[\alpha]} I_{[\gamma]} = 0$ if $[\alpha] \neq [\gamma].$
\end{proof}

\medskip

\begin{definition}\rm
We call the {\it center} of a split Malcev-Poisson-Jordan algebra $P$ to the set $${\mathcal Z}(P) := \Bigl\{v \in P : [v,P] = vP =  0\Bigr\}.$$
\end{definition}

\begin{corollary}\label{coro1}
If ${\mathcal Z}(P)=0$ and $H = \sum_{\alpha \in \Lambda}\Bigl([P_{\alpha},P_{-\alpha}] + P_{\alpha}P_{-\alpha}\Bigr),$ then $P$ is the direct sum of the ideals given in Theorem \ref{teo1}-i), $$P = \bigoplus\limits_{[\alpha] \in \Lambda/\sim} I_{[\alpha]},$$ being $\bigl[I_{[\alpha]},I_{[\gamma]}\bigr] = I_{[\alpha]} I_{[\gamma]} = 0$ if $[\alpha] \neq [\gamma].$
\end{corollary}

\begin{proof}
From $H = \sum_{\alpha \in \Lambda}\Bigl([P_{\alpha},P_{-\alpha}] + P_{\alpha}P_{-\alpha}\Bigr)$ it is clear that $P = \sum_{[\alpha] \in \Lambda /\sim} I_{[\alpha]}$. In order to verify the direct character of the sum, consider $$v \in I_{[\alpha]} \cap \bigl(\sum\limits_{[\gamma] \in \Lambda
/\sim, [\gamma] \neq [\alpha]} I_{[\gamma]}\bigr).$$ Since $v
\in I_{[\alpha]}$, the fact $\bigl[I_{[\alpha]}, I_{[\gamma]}\bigr] = I_{[\alpha]} I_{[\gamma]} = 0$ in case $[\alpha] \neq [\gamma]$ gives us
\begin{equation}\label{center1}
\bigl[v, \sum\limits_{[\gamma] \in \Lambda/\sim, [\gamma] \neq
[\alpha]} I_{[\gamma]} \bigr] = v \bigl(\sum\limits_{[\gamma] \in \Lambda/\sim, [\gamma] \neq [\alpha]} I_{[\gamma]} \bigr) =0.
\end{equation}
Taking now into account $v \in \sum\limits_{[\gamma] \in \Lambda /\sim, [\gamma] \neq [\alpha]} I_{[\gamma]}$, the same above fact implies
\begin{equation}\label{center2}
\bigl[v,I_{[\alpha]}\bigr] = v  I_{[\alpha]} =0.
\end{equation}
From Equations \eqref{center1} and \eqref{center2} we get $v \in {\mathcal Z}(P)=0$. From here, $P = \bigoplus\limits_{[\alpha] \in \Lambda/\sim} I_{[\alpha]}$.
\end{proof}

%%%%%%%%%%%%%%%%%%%%%%%%%%%%%%%%%%%%%%%%%%%%%%%%%%%%%%%%%%%%%%%%%
%%%%%%%%%%%%%%%%%%%%%%%%%%%%%%%%%%%%%%%%%%%%%%%%%%%%%%%%%%%%%%%%%
\section{The simple components}
%%%%%%%%%%%%%%%%%%%%%%%%%%%%%%%%%%%%%%%%%%%%%%%%%%%%%%%%%%%%%%%%%
%%%%%%%%%%%%%%%%%%%%%%%%%%%%%%%%%%%%%%%%%%%%%%%%%%%%%%%%%%%%%%%%%

In this section we are interested in studying under which
conditions a split Malcev-Poisson-Jordan algebra decomposes as the direct sum of the family of its simple ideals. From now on, we assume $\Lambda$ is symmetric (in the sense $\alpha \in \Lambda$ implies $-\alpha \in \Lambda$) and ${\rm char}({\mathbb K})=0$.

\begin{lemma}\label{lema5}
Let $P = H \oplus (\bigoplus_{\alpha \in \Lambda} P_{\alpha})$ be a split Malcev-Poisson-Jordan algebra. If $I$ is an ideal of $P$ then $I = (I \cap H ) \oplus \bigl(\bigoplus_{\alpha \in \Lambda} (I \cap P_{\alpha})\bigr).$
\end{lemma}
\begin{proof}
Consider the set $Ad := \{{\rm ad}(h): h \in H\},$ being ${\rm ad}(h)(x) := [h,x]$ for any $x \in P$. The result follows from the fact that $Ad$ is a commuting set of diagonalizable endomorphisms, and $I$ is invariant under this set.
\end{proof}

\begin{lemma}\label{lema4}
Let $P = H \oplus (\bigoplus_{\alpha \in \Lambda} P_{\alpha})$ be a split Malcev-Poisson-Jordan algebra with ${\mathcal Z}(P)=0$. If $I$ is an ideal of $P$ such that $I \subset H$ then $I=\{0\}$.
\end{lemma}

\begin{proof}
Suppose that $I$ is an ideal of $P$ such that $I \subset H$. Since  $H$ is abelian
we have $ [I,H] \subset [H,H] = 0 $ and  $IH \subset HH = 0$. The fact $$\bigl[I,\bigoplus\limits_{\alpha \in \Lambda} P_{\alpha}\bigr] + I  \bigl(\bigoplus\limits_{\alpha\in \Lambda} P_{\alpha}\bigr)   \subset I \subset H$$ joint with Lemma \ref{elprimero} implies
\begin{equation*}
\bigl[I, \bigoplus\limits_{\alpha \in \Lambda} P_{\alpha} \bigr]  + I  \bigl(\bigoplus\limits_{\alpha \in \Lambda} P_{\alpha} \bigr)  \subset H \cap \bigl(\bigoplus\limits_{\alpha \in \Lambda} P_{\alpha} \bigr)= 0,
\end{equation*}
thus $\bigl[I,  P  \bigr]= 0$ and  $IP = 0$. Therefore, we have showed $I \subset \mathcal{Z}(P)=0$.
\end{proof}

\noindent Let us introduce the concepts of root-multiplicativity and maximal length in the framework of split Malcev-Poisson-Jordan algebras, in a similar way to the ones for split Lie algebras, split Malcev algebras, split Poisson algebras and split Lie-Rinehart algebras (see \cite{YoLieRinehart, super, YoPoisson, YoMalcev} for these notions and examples).

\begin{definition}\rm
We say that a split Malcev-Poisson-Jordan algebra $P$ is {\it root-multiplicative} if given $\alpha, \beta  \in \Lambda$ such that either $\alpha \star \beta \in \Lambda$ or $\theta_{\alpha} \star \beta \in \Lambda$, then either $[P_{\alpha},P_{\beta}] \neq 0$ and $P_{\alpha}P_{\beta} \neq 0$, or $[[P_{\alpha},P_{-\alpha}], P_{\beta}] \neq 0$, respectively.
\end{definition}

\begin{definition}\rm
A split Malcev-Poisson-Jordan algebra $P$ is of {\it maximal length} if $\dim P_{\alpha} =1$ for any $\alpha \in \Lambda$.
\end{definition}

We would like to note that the above concepts appear in a natural way in the study of split Malcev-Poisson-Jordan algebras. For instance, any Poisson structure associate to a semisimple finite dimensional Lie algebra gives rise to a root-multiplicative and of maximal length split Poisson algebra. We also have, in the infinite-dimensional setting, that any Poisson structure ${\mathcal P}$ defined either on a semisimple separable $L^*$-algebra \cite{Schue1, Schue2} or on a semisimple locally finite split Lie algebra \cite{Stumme} necessarily makes $\mathcal{P}$ a root-multiplicative and of maximal length split Poisson algebra.

%\noindent We consider the next definition taken from the theory of split Lie algebras and split Lie triple systems \cite{Integra, Neeb, Stumme}.
%\begin{definition}
%We say that a nonzero root $\alpha$ of a split Malcev-Poisson-Jordan algebra $P$ is {\it abelian} if there exists $0 \neq e_{\alpha} \in P_{\alpha}$ such that $\{e_{\alpha}, P_{-\alpha}\} = e_{\alpha}P_{-\alpha} = 0$.
%\end{definition}

%\noindent We are interested in split Malcev-Poisson-Jordan algebras with no abelian nonzero roots. As examples of split Malcev-Poisson-Jordan algebras satisfying this fact we have the well-known non-Lie simple Malcev algebra ${{\frak C}_0}$ (see the  multiplication table in \cite[Section 6]{Sagle2}) and so all of  the finite dimensional semisimple Malcev algebras (over an algebraically closed field). We also have the semisimple separable $L^*$-algebras and the semisimple locally finite  split Lie algebras over a field of characteristic zero. Actually, observe that all of these examples are root-multiplicative split Malcev-Poisson-Jordan algebras contains no abelian nonzero roots.

\begin{theorem}\label{teo3}
Let $P$ be a root-multiplicative split Malcev-Poisson-Jordan algebra of ma\-ximal length. It holds that $P$ is simple if and only if it has all its non-zero roots connected and $H = \sum_{\alpha \in \Lambda}\Bigl([P_{\alpha}, P_{-\alpha}] + P_{\alpha} P_{-\alpha}\Bigr)$.
\end{theorem}

\begin{proof}
If $P$ is simple, the necessary condition was proved in Theorem \ref{teo1}-ii). To prove the converse, consider $I$ a non-zero ideal of $P$. By Lemmas \ref{lema5} and \ref{lema4} we can write $$I = (I \cap H) \oplus (\bigoplus_{\alpha \in \Lambda_I} (I \cap P_{\alpha}))$$ with $\Lambda_I := \{\alpha \in \Lambda : I \cap P_{\alpha} \neq 0\}$ and $\Lambda_I \neq \emptyset$. By maximal length we also can write $\Lambda_I = \{\alpha \in \Lambda : P_{\alpha} \subset I\}.$

Let us show that $H \subset I$. We take $\alpha_0 \in \Lambda_I,$ so $P_{\alpha_0} \subset I$. Now, for any $\beta \in \Lambda \setminus \{2\alpha_0,-2\alpha_0\}$, the fact that $\alpha_0$ and $\beta$ are connected gives us a connection $\{\gamma_1,\dots,\gamma_r\}$ from $\alpha_0$ to $\beta$ such that $$\gamma_1 = \alpha_0, \gamma_1 \star \gamma_2,(\gamma_1 \star \gamma_2) \star \gamma_3,\dots, (\cdots((\gamma_1 \star \gamma_2) \star \gamma_3)\star \cdots) \star \gamma_{r-1} \in   \Lambda \cup \Theta_{\Omega}$$
$$(\cdots((\gamma_1 \star \gamma_2) \star \gamma_3) \star \cdots) \star \gamma_r \in \{\beta,-\beta\},$$ and by maximal length and the root-multiplicativity of $P$ we get $$[[\dots[[P_{\alpha_0},P_{\gamma_2}],P_{\gamma_3}],\dots],P_{\gamma_r}] = P_{\epsilon \beta}, \mbox{ with} \hspace{0.2cm} \epsilon \in \{1,-1\}.$$ From here, we deduce that either $P_{\beta} \subset I$ or $P_{-\beta} \subset I.$ In both cases $[P_{\beta},P_{-\beta}] \subset I$ and $P_{\beta}P_{-\beta} \subset I.$ From here, the fact $H = \sum_{\beta \in \Lambda} ([P_{\beta},P_{-\beta}] + P_{\beta}P_{-\beta})$ implies $H \setminus \Bigl\{ [P_{2\alpha_0},P_{-2\alpha_0}] + P_{2\alpha_0}P_{-2\alpha_0} \Bigr\} \subset I$. Since $\alpha_0 \in \Lambda_I,$ by maximal length and Lemma \ref{elprimero}-iii) we get $P_{2\alpha_0} = P_{\alpha_0}P_{\alpha_0} \subset I,$ and consequently we can assert
\begin{equation}\label{eq7}
H \subset I.
\end{equation}
Given now any $\alpha \in \Lambda$, the fact $\alpha \neq 0$ and Equation \eqref{eq7} show $P_{\alpha}=[H,P_{\alpha}]  \subset I$. We conclude $I=P$ and therefore $P$ is simple as required.
\end{proof}

\begin{theorem}
Let $P$ be a root-multiplicative split Malcev-Poisson-Jordan algebra of maxi\-mal length satisfying $\mathcal{Z}(P)=0$ and $H = \sum\limits_{\alpha \in \Lambda}\Bigl([P_{\alpha},P_{-\alpha}] + P_{\alpha}  P_{-\alpha}\Bigr)$. Then $$P =\bigoplus\limits_{[\alpha] \in \Lambda/\sim} I_{[\alpha]},$$ where any $I_{[\alpha]}$ is a simple ideal having its root system $\Lambda_{I_{[\alpha]}}$  with all of its roots connected and satisfying $[I_{[\alpha]},I_{[\gamma]}] = I_{[\alpha]} I_{[\gamma]} = 0$ if $[\alpha] \neq [\gamma]$.
\end{theorem}
\black

\begin{proof}
By Corollary \ref{coro1}, $$P = \bigoplus\limits_{[\alpha] \in \Lambda/\sim} I_{[\alpha]}$$ is the direct sum of the ideals
$$I_{[\alpha]} = I_{H,[\alpha]} \oplus V_{[\alpha]},$$ where $$I_{H,[\alpha]} = span_{\mathbb K} \Bigl\{[P_{\beta},P_{-\beta}] + P_{\beta}P_{-\beta} : \beta \in [\alpha]\Bigr\},$$ having any $I_{[\alpha]}$ as root system $\Lambda_{I_{[\alpha]}} = [\alpha].$ In order to apply Theorem \ref{teo3} to each $I_{[\alpha]}$, observe that the fact $\Lambda_{I_{[\alpha]}} = [\alpha]$ gives us easily that $\Lambda_{I_{[\alpha]}}$ has all of its elements connected through connections contained in $\Lambda_{I_{[\alpha]}}$. We also have that any $I_{[\alpha]}$ is root-multiplicative as consequence of the root-multiplicativity of $P$. Clearly $I_{[\alpha]}$ is of maximal length because $P$ is itself of maximal length, and finally ${\mathcal Z}_{I_{[\alpha]}}(I_{[\alpha]}) = 0$ (where ${\mathcal Z}_{I_{[\alpha]}}(I_{[\alpha]})$ denotes the center of $I_{[\alpha]}$ in itself) as consequence of ${\mathcal Z}(P)=0$ and $\bigl[I_{[\alpha]}, I_{[\gamma]}\bigr] = I_{[\alpha]}I_{[\gamma]} = 0$ if $[\alpha] \neq [\gamma]$ (Corollary \ref{coro1}). We can apply Theorem \ref{teo3} to any $I_{[\alpha]}$ which allows us to conclude $I_{[\alpha]}$ is simple. It is clear  that the decomposition $P = \bigoplus_{[\alpha] \in \Lambda/\sim} I_{[\alpha]}$ satisfies the assertions of the theorem, completing the proof.
\end{proof}

\medskip

%{\bf Acknowledgment.} I would like to thank the referee for the detailed reading of this work and for the suggestions which have improved the final version of the same.In the following $P$ denotes a split Malcev-Poisson-Jordan algebra with root system $\Lambda$, and $P = H \oplus (\bigoplus_{\alpha \in \Lambda} P_{\alpha})$ is its corresponding root decomposition.

\medskip

%{\bf Acknowledgment.} I would like to thank the referee for the detailed reading of this work and for the suggestions which have improved the final version of the same.


\begin{thebibliography}{9}

\bibitem{MPJ} Ait Ben Haddou, M., Benayadi, S., Boulmane, S. (2016). Malcev-Poisson-Jordan algebras. {\it J. Algebra Appl.} {\bf 15}(9), 1650159, 26 pp. DOI: 10.1142/S0219498816501590

\bibitem{YoLieRinehart} Albuquerque, H., Barreiro, E., Calder\'on, A.J., S\'anchez, J.M. (2021). Split Lie-Rinehart algebras. {\it J. Algebra  Appl.} {\bf 20}(9), 2150164. DOI: 10.1142/S0219498821501644

\bibitem{super} Calder\'on, A.J. (2008). On split Lie algebras with symmetric root systems. {\it Proc. Indian Acad. Sci. (Math. Sci.)} {\bf 118}(3), 351-356. DOI: 10.1007/s12044-008-0027-3

\bibitem{YoPoisson} Calder\'on, A.J. (2012). On the structure of split non-commutative Poisson algebras. {\it Linear Multilinear Algebra} {\bf 60}(7), 775-785. DOI: 10.1080/03081087.2012.661428

\bibitem{YoMalcev} Calder\'on, A.J., Forero, M., S\'anchez, J.M. (2012). Split Malcev algebras. {\it Proc. Indian Acad. Sci. (Math. Sci.)} {\bf 122}(2), 181-187. DOI: 10.1007/s12044-012-0070-y

\bibitem{Malcev_Poisson} S\'anchez, J.M. (2021). On split Malcev Poisson algebras. {\it Siberian Math.} {\bf 62}, 511–520. DOI: 10.1134/S0037446621030149

\bibitem{Schue1} Schue, J.R. (1960). Hilbert Space methods in the theory of Lie algebras. {\it Trans. Amer. Math. Soc.} {\bf 95}, 69-80. DOI: 10.1090/S0002-9947-1960-0117575-1

\bibitem{Schue2} Schue, J.R. (1961). Cartan decompositions for $L^{*}$-algebras. {\it Trans. Amer. Math. Soc.} {\bf 98}, 334-349. DOI: 10.2307/1993500

\bibitem {Stumme} Stumme, N. (1999). The structure of Locally Finite Split Lie Algebras. {\it J. Algebra} {\bf 220}, 664-693. DOI: 10.1006/jabr.1999.7978

\end{thebibliography}
\end{document}